\documentclass[12pt,psamsfonts,reqno]{amsart}
\usepackage{amssymb}

\numberwithin{equation}{section}

\newcommand{\pup}[1]{\textup{(}#1\textup{)}}

\theoremstyle{plain}

\newtheorem{lemma}{Lemma}[section]
\newtheorem{theorem}[lemma]{Theorem}
\newtheorem{proposition}[lemma]{Proposition}
\newtheorem{corollary}[lemma]{Corollary}
\newtheorem{problem}{Problem}
\newtheorem{question}{Question}

\theoremstyle{definition}
\newtheorem{definition}[lemma]{Definition}

\theoremstyle{remark}
\newtheorem{remark}[lemma]{Remark}

\newcommand{\qedsc}{{\qed}~{\rm Claim.}}

\newcommand{\set}[1]{\{#1\}}

\newcommand{\setm}[2]{\set{{#1}\mid{#2}}}

\newcommand{\famm}[2]{({#1}\mid{#2})}

\newcommand{\seq}[1]{\left\langle{#1}\right\rangle}

\newcommand{\utr}{\trianglelefteq}
\newcommand{\utrl}{\utr_{\mathrm{left}}}
\newcommand{\utrr}{\utr_{\mathrm{right}}}
\newcommand{\jh}{join-ho\-mo\-mor\-phism}
\newcommand{\mh}{meet-ho\-mo\-mor\-phism}
\newcommand{\res}{\mathbin{\restriction}}
\newcommand{\je}{join-em\-bed\-ding}
\newcommand{\ol}[1]{\overline{#1}}
\newcommand{\les}{\leqslant}
\newcommand{\ges}{\geqslant}

\DeclareMathOperator{\Rank}{Rank}
\DeclareMathOperator{\Self}{Self}
\DeclareMathOperator{\Rel}{Rel}
\DeclareMathOperator{\PSelf}{PSelf}
\newcommand{\RR}{{}_RR}

\newcommand{\Selfltwo}{\operatorname{\Self_{\les2}}}
\newcommand{\Selftwo}{\operatorname{\Self_2}}
\newcommand{\Selffin}{\operatorname{\Self_{\mathrm{fin}}}}
\DeclareMathOperator{\card}{card}
\DeclareMathOperator{\Eq}{Eq}
\newcommand{\Eqltwo}{\operatorname{\Eq^{\les2}}}
\newcommand{\Eqtwo}{\operatorname{\Eq^2}}
\newcommand{\Eqfin}{\operatorname{\Eq^{\mathrm{fin}}}}
\DeclareMathOperator{\rng}{rng}
\DeclareMathOperator{\Ker}{Ker}
\DeclareMathOperator{\Sym}{Sym}
\DeclareMathOperator{\End}{End}
\DeclareMathOperator{\Sub}{Sub}

\DeclareMathOperator{\Endf}{\operatorname{\End_{\mathrm{fin}}}}
\DeclareMathOperator{\Subdf}{\operatorname{\Sub_{\mathrm{fin}}}}
\DeclareMathOperator{\Subuf}{\operatorname{\Sub^{\mathrm{fin}}}}
\newcommand{\op}{{\mathrm{op}}}
\newcommand{\id}{{\mathrm{id}}}
\newcommand{\Pow}{{\mathfrak{P}}}
\newcommand{\Fin}[1]{[{#1}]^{<\omega}}
\DeclareMathOperator{\codim}{codim}
\newcommand{\eps}{\varepsilon}
\newcommand{\one}{\mathbf{1}}

\newcommand{\es}{\varnothing}
\newcommand{\FF}{\mathbb{F}}
\newcommand{\QQ}{\mathbb{Q}}
\newcommand{\ZZ}{\mathbb{Z}}

\newcommand{\rT}{\mathrm{T}}
\newcommand{\rQ}{\mathrm{Q}}
\newcommand{\rM}{\mathrm{M}}
\newcommand{\rS}{\mathrm{S}}
\newcommand{\rC}{\mathrm{C}}
\DeclareMathOperator{\rE}{E}
\DeclareMathOperator{\rF}{F}

\begin{document}
\baselineskip=17pt

\title[Endomorphism semigroups]{Embedding properties of endomorphism semigroups}
\author{Jo\~ao Ara\'ujo}
\address{Universidade Aberta,
Rua da Escola Polit\'ecnica, 147, 1269-001 Lisboa \and Centro de \'Algebra da Universidade de Lisboa, Av. Gama Pinto, 2, 1649--003 Lisboa, Portugal }
\email{jaraujo@ptmat.fc.ul.pt}
\urladdr{http://caul.cii.fc.ul.pt/pt/member10.html}

\author{Friedrich Wehrung}
\address{LMNO, CNRS UMR 6139\\
Universit\'e de Caen\\
Campus 2\\
D\'{e}partement de Math\'{e}matiques\\
BP 5186\\
14032 Caen cedex\\
France}
\email{wehrung@math.unicaen.fr}
\urladdr{http://www.math.unicaen.fr/$\scriptstyle\sim$wehrung}

\thanks{The first author was partially supported by FCT and FEDER, Project POCTI-ISFL-1-143 of Centro de Algebra da Universidade de Lisboa, and by FCT and PIDDAC through the project PTDC/MAT/69514/2006}

\date{\today}

\begin{abstract}
Denote by $\PSelf\Omega$ (resp., $\Self\Omega$) the partial (resp., full) transformation monoid over a set~$\Omega$, and by $\Sub V$ (resp., $\End V$) the collection of all subspaces (resp., endomorphisms) of a vector space~$V$. We prove various results that imply the following:
\begin{itemize}
\item[(1)] If $\card\Omega\ges2$, then $\Self\Omega$ has a semigroup embedding into the dual of $\Self\Gamma$ if{f} $\card\Gamma\ges2^{\card\Omega}$. In particular, if $\Omega$ has at least two elements, then there exists no semigroup embedding from~$\Self\Omega$ into the dual of~$\PSelf\Omega$.

\item[(2)] If~$V$ is infinite-dimensional, then there is no embedding from $(\Sub V,+)$ into $(\Sub V,\cap)$ and no embedding from $(\End V,\circ)$ into its dual semigroup.

\item[(3)] Let~$F$ be an algebra freely generated by an infinite subset~$\Omega$. If~$F$ has less than~$2^{\card\Omega}$ operations, then~$\End F$ has no semigroup embedding into its dual. The cardinality bound~$2^{\card\Omega}$ is optimal.

\item[(4)] Let $F$ be a free left module over a left $\aleph_1$-n\oe therian ring (i.e., a ring without strictly increasing chains, of length~$\aleph_1$, of left ideals). Then $\End F$ has no semigroup embedding into its dual.
\end{itemize}
(1) and (2) above solve questions proposed by B.\,M. Schein and G.\,M. Bergman. We also formalize our results in the settings of algebras endowed with a notion of independence (in particular \emph{independence algebras}).
\end{abstract}

\maketitle

\section{Introduction}\label{S:Intro}
A \emph{\pup{partial} function} on a set~$\Omega$ is a map from a subset of~$\Omega$ to~$\Omega$. The composition $g\circ f$ of partial functions~$f$, $g$ on~$\Omega$ is a partial function, with domain the set of all~$x$ in the domain of~$f$ such that $f(x)$ belongs to the domain of~$g$. The set~$\PSelf\Omega$ of all partial functions on~$\Omega$ is a monoid under composition. Denote by~$\Self\Omega$ the submonoid of~$\PSelf\Omega$ consisting of all endomaps of~$\Omega$. The \emph{dual}~$S^\op$ of a semigroup (resp., monoid)~$S$ with multiplication~$\cdot$ is defined as the semigroup (resp., monoid) with the same underlying set as~$S$ and the multiplication~$*$ defined by the rule $x*y=y\cdot x$ for all $x,y\in S$. A \emph{dual automorphism} (resp., a \emph{dual embedding}) of~$S$ is an isomorphism (resp., embedding) from~$S$ to~$S^\op$.

In the present paper, we solve the following three questions:

\begin{question}\label{Qu:1}
Suppose that~$\Omega$ is infinite. Does $\Self\Omega$ have a dual embedding?
\end{question}

\begin{question}\label{Qu:2}
Suppose that~$\Omega$ is infinite. Does $\PSelf\Omega$ have a dual embedding?
\end{question}

\begin{question}\label{Qu:3}
Does the endomorphism monoid of an infinite-dimensional vector space have a dual embedding?
\end{question}

Question~\ref{Qu:1} originates in an earlier version of a preprint by George Berg\-man~\cite{Berg} and Questions~\ref{Qu:1} and~\ref{Qu:2} were proposed by Boris Schein in September 2006 while he gave a course on semigroups at the Center of Algebra of the University of Lisbon. After learning some of the results of the present paper, proved by the second author, that implied a negative answer to Question~\ref{Qu:1}, Bergman changed~\cite{Berg} and subsequently asked Question~\ref{Qu:3}. This question was solved by the second author as well. The original solution of Question~\ref{Qu:1} was obtained \emph{via} an analogue of Theorem~\ref{T:Selfmaps1} but with a non-optimal bound; in our present formulation of that theorem, the optimal bound $2^{\card\Omega}$ is proved. Furthermore, the similarity of the methods used in the (negative) solutions of all these questions lead us to the investigation of more general classes of algebras where similar negative results would hold, for example $M$-acts or modules. 

The road to the latter goal is opened as follows. As both $\Self\Omega$ and $\End V$ are endomorphism monoids of universal algebras, we move forward to identify more general classes of universal algebras whose endomorphism monoids cannot be embedded into their dual. In particular, this is the case for the free objects in any nontrivial variety with small enough similarity type (Theorem~\ref{T:popular}), but not necessarily for all free $M$-acts for suitable monoids~$M$ (Theorem~\ref{T:AlmostDual}). In Section~\ref{S:IndepAlg}, we introduce a rather large class of algebras whose endomorphism monoids cannot be embedded into their dual, called \emph{SC-ranked} algebras (Definition~\ref{D:Sranked} and Corollary~\ref{C:RkA=RkBindep}). These algebras arise from the study of algebras endowed with a notion of \emph{independence} (see Section~\ref{S:CSM}). This gives, for example, new results about $M$-acts for monoids~$M$ without large left divisibility antichains (Theorem~\ref{T:FM(Omega)Srk}), in particular for \emph{$G$-sets} (Corollary~\ref{C:NoEmbGsets}), but also for modules over rings satisfying weak n\oe therianity conditions (Corollary~\ref{C:Srkmodule}).

Denote by~$\Sub V$ (resp., $\End V$) the collection of all subspaces (resp., endomorphisms) of a vector space~$V$. Our results imply the following:

\begin{itemize}
\item (cf. Corollary~\ref{C:Selfmaps2}) \emph{Let~$\Omega$ and~$\Gamma$ be sets with $\card\Omega\ges2$. Then~$\Self\Omega$ has a semigroup embedding into~$(\Self\Gamma)^\op$ if{f} $\card\Gamma\ges2^{\card\Omega}$.}

\item (cf. Theorems~\ref{T:NoLattEmb} and~\ref{T:EndIneq}) \emph{Let~$V$ and~$W$ be right vector spaces over division rings~$K$ and~$F$, respectively, with~$V$ infinite-dimensional. If there exists an embedding either from $(\Sub V,+)$ into $(\Sub W,\cap)$ or from $(\End V,\circ)$ to $(\End W,\circ)^\op$, then $\dim W\ges(\card K)^{\dim V}$.}

\item (cf. Theorems~\ref{T:popular} and~\ref{T:AlmostDual}) \emph{Let $\mathcal{V}$ be a variety of algebras, not all reduced to a singleton, in a similarity type~$\Sigma$, and let~$\Omega$ be an infinite set. If $\card\Sigma<2^{\card\Omega}$, then the endomorphism semigroup of the free algebra on~$\Omega$ in~$\mathcal{V}$ has no dual embedding. The cardinality bound $2^{\card\Omega}$ is optimal, even for $M$-acts for a suitably chosen monoid~$M$}.

\item (cf. Theorem~\ref{C:Srkmodule}) \emph{Let $F$ be a free left module over a ring in which there is no strictly increasing $\aleph_1$-sequence of left ideals. Then the semigroup $\End F$ has no dual embedding}.

\end{itemize}

In Section~\ref{S:Concl}, we formulate a few concluding remarks and open problems.

\section{Basic concepts}
For a nonzero cardinal~$\kappa$, we put $\kappa-1=\card(\Omega\setminus\set{p})$, for any set~$\Omega$ of cardinality~$\kappa$ and any~$p\in\Omega$ (so $\kappa-1=\kappa$ in case~$\kappa$ is infinite).
We denote by~$\Pow(\Omega)$ the powerset of a set~$\Omega$, and by $\Fin{\Omega}$ the set of all finite subsets of~$\Omega$. We put
 \[
 \Ker f=\setm{(x,y)\in\Omega\times\Omega}{f(x)=f(y)}\,,
 \ \text{for any function }f \text{ with domain }\Omega.
 \]
We also denote by $\rng f$ the range of~$f$. We denote the partial operation of disjoint union by~$\sqcup$.

We denote by $\Eq\Omega$ the lattice of all equivalence relations on~$\Omega$ under inclusion, and we denote by~$[x]_{\theta}$ the~$\theta$-block of any element~$x\in\Omega$, for each~$\theta\in\Eq\Omega$. We put
 \begin{align*}
 \Eqltwo\Omega&=\setm{\theta\in\Eq\Omega}{\card(\Omega/{\theta})\les2},\\
 \Eqtwo\Omega&=\setm{\theta\in\Eq\Omega}{\card(\Omega/{\theta})=2},\\
 \Eqfin\Omega&=\setm{\theta\in\Eq\Omega}{\Omega/{\theta}\text{ is finite}}.
 \end{align*}
The monoid~$\Self\Omega$ has the following subsets, the first three of which are also subsemigroups:
 \begin{align*}
 \Sym\Omega&=\setm{f\in\Self\Omega}{f\text{ is bijective}},\\
 \Selfltwo\Omega&=\setm{f\in\Self\Omega}{\card(\rng f)\les2},\\
 \Selffin\Omega&=\setm{f\in\Self\Omega}{\rng f\text{ is finite}},\\
 \Selftwo\Omega&=\setm{f\in\Self\Omega}{\card(\rng f)=2}.
 \end{align*}
We put $\ker f=f^{-1}\set{0}$ (the usual \emph{kernel} of~$f$), for any homomorphism~$f$ of abelian groups.
For a right vector space~$V$ over a division ring~$K$, we denote by~$\Subdf V$ (resp., $\Subuf V$) the sublattice of~$\Sub V$ consisting of all finite-dimensional (resp., finite-codimensional) subspaces of~$V$. Furthermore, we denote by~$\Endf V$ the semigroup of all endomorphisms with finite-dimensional range of~$V$. In particular, the elements of~$\Subuf V$ are exactly the kernels of the elements of~$\Endf V$.

\section{Embeddings between semigroups of endomaps}
\label{S:Endomaps}

For any~$f\in\Self\Omega$, denote by~$f^{-1}$ the endomap of the powerset~$\Pow(\Omega)$ that sends every subset of~$\Omega$ to its inverse image under~$f$. The assignment $\Self\Omega\to\Self\Pow(\Omega)$, $f\mapsto f^{-1}$ defines a monoid embedding from $\Self\Omega$ into $(\Self\Pow(\Omega))^\op$. Moreover, both~$\Self1$ and~$\Self\es$ are the one-element monoid, which is self-dual. For larger sets the following theorem says that the assignment $f\mapsto f^{-1}$ described above is optimal in terms of size.

\begin{theorem}\label{T:Selfmaps1}
Let $\Omega$ and $\Gamma$ be sets with $\card\Omega\ges2$. If there exists a semigroup embedding from $\Selfltwo\Omega$ into $(\Self\Gamma)^\op$, then $\card\Gamma\ges2^{\card\Omega}$.
\end{theorem}

We prove Theorem~\ref{T:Selfmaps1} in a series of lemmas. Assuming an embedding from $\Selfltwo\Omega$ into $(\Self\Gamma)^\op$,  Lemma~\ref{L:epskerim} is used to associate the kernel of a function in~$\Selfltwo\Omega$ with the range of its image under the embedding. As any two distinct members of~$\Eqtwo\Omega$ join to the coarse equivalence relation in an `effective' way (Lemma~\ref{L:gf}), this will give, in Lemma~\ref{L:mumeet}, a partition of a suitable subset of~$\Gamma$ with many classes. Proving that each of these classes has at least two elements is the object of Lemmas~\ref{L:cardKges2} and~\ref{L:jumpeka}; this will give the final estimate.

\begin{lemma}\label{L:gf}
Let~$\alpha$ and~$\beta$ be distinct elements in~$\Eqtwo\Omega$. Then there are idempotent maps $f,g\in\Selftwo\Omega$ such that $\Ker f=\alpha$, $\Ker g=\beta$, and $f\circ g$ is constant.
\end{lemma}

\begin{proof}
As~$\alpha\neq\beta$, we can write $\Omega/{\alpha}=\set{A_0,A_1}$ and $\Omega/{\beta}=\set{B_0,B_1}$ with both $A_0\cap B_0$ and $A_0\cap B_1$ nonempty. Pick $b_i\in A_0\cap B_i$, for $i<2$, and pick $a\in A_1$. Define idempotent endomaps~$f$ and~$g$ of~$\Omega$ by the rule
 \[
 f(x)=\begin{cases}
 b_0&(x\in A_0),\\
 a&(x\in A_1),
 \end{cases}\quad
 g(x)=\begin{cases}
 b_0&(x\in B_0),\\
 b_1&(x\in B_1),
 \end{cases}\qquad
 \text{for all }x\in\Omega.
 \]
Then $\Ker f=\alpha$, $\Ker g=\beta$, and $f\circ g$ is the constant function with value~$b_0$.
\end{proof}

Now let $\eps\colon\Selfltwo\Omega\hookrightarrow(\Self\Gamma)^\op$ be a semigroup embedding.

\begin{lemma}\label{L:epskerim}
$\Ker f\subseteq\Ker g$ implies that $\rng\eps(g)\subseteq\rng\eps(f)$, for all $f,g\in\Selfltwo\Omega$.
\end{lemma}

\begin{proof}
There exists $h\in\Selfltwo\Omega$ such that $g=h\circ f$. Thus $\eps(g)=\eps(f)\circ\eps(h)$ and the conclusion follows.
\end{proof}

Lemma~\ref{L:epskerim} makes it possible to define a map
 \[
 \mu\colon\Eqltwo\Omega\to\Pow(\Gamma)\setminus\set{\es}
 \]
by the rule $\mu(\Ker f)=\rng\eps(f)$, for each $f\in\Selfltwo\Omega$. 

\begin{lemma}\label{L:muantemb}
$\alpha\subseteq\beta$ if{f} $\mu(\beta)\subseteq\mu(\alpha)$, for all $\alpha,\beta\in\Eqltwo\Omega$.
\end{lemma}

\begin{proof}
The direction from the left to the right (i.e., the map~$\mu$ is \emph{antitone}) follows from Lemma~\ref{L:epskerim}. Now assume that $\mu(\beta)\subseteq\mu(\alpha)$. There are idempotent $f,g\in\Selfltwo\Omega$ such that $\alpha=\Ker f$ and $\beta=\Ker g$. As $\rng\eps(g)\subseteq\rng\eps(f)$ and~$\eps(f)$ is idempotent, $\eps(f)\circ\eps(g)=\eps(g)$, that is, $\eps(g\circ f)=\eps(g)$, and thus, as~$\eps$ is one-to-one, $g\circ f=g$, and therefore $\Ker f\subseteq\Ker g$.
\end{proof}

Let $\one=\Omega\times\Omega$ denote the coarse equivalence relation on~$\Omega$.

\begin{lemma}\label{L:mumeet}
$\mu(\alpha)\cap\mu(\beta)=\mu(\one)$, for all distinct $\alpha,\beta\in\Eqtwo\Omega$.
\end{lemma}

\begin{proof}
It follows from Lemma~\ref{L:gf} that there are idempotent $f,g\in\Self\Omega$ such that $\Ker f=\alpha$, $\Ker g=\beta$, and~$f\circ g$ is constant.

Let $x\in\mu(\alpha)\cap\mu(\beta)$. This means that~$x$ belongs to both~$\rng\eps(f)$ and~$\rng\eps(g)$, hence, as both~$\eps(f)$ and~$\eps(g)$ are idempotent, that it is fixed by both these maps, hence that it is fixed by their composite, $\eps(g)\circ\eps(f)=\eps(f\circ g)$, hence it lies in the range of that composite, which, as~$f\circ g$ is a constant function, is~$\mu(\one)$.

So we have proved that $\mu(\alpha)\cap\mu(\beta)$ is contained in~$\mu(\one)$. As the converse inequality follows from Lemma~\ref{L:epskerim}, the conclusion follows.
\end{proof}

Denote by $k_x$ the constant function on~$\Omega$ with value~$x$, for each~$x\in\Omega$. Hence $\mu(\one)=\rng\eps(k_x)$.

\begin{lemma}\label{L:cardKges2}
The set $\mu(\one)$ has at least two elements.
\end{lemma}

\begin{proof}
Otherwise, $\mu(\one)=\set{z}$ for some~$z\in\Gamma$, and so~$\eps(k_x)$ is the constant function on~$\Gamma$ with value~$z$, for each~$x\in\Omega$. As~$\eps$ is one-to-one, this implies that~$\Omega$ has at most one element, a contradiction.
\end{proof}

\begin{lemma}\label{L:jumpeka}
The set $\rng\eps(e)\setminus\mu(\one)$ has at least two elements, for each idempotent~$e\in\Selftwo\Omega$.
\end{lemma}

\begin{proof}
Let~$\rng e=\set{x,y}$. It follows from Lemmas~\ref{L:epskerim} and~\ref{L:muantemb} that $\rng\eps(e)$ properly contains~$\mu(\one)$. Suppose that $\rng\eps(e)\setminus\mu(\one)=\set{t}$, for some~$t\in\Gamma$.

For elements~$a$ and~$b$ in a semigroup~$S$, let $a\sim b$ hold, if there are elements $x_1,x_2,y_1,y_2\in S$ such that $a=x_1b=bx_2$ and $b=y_1a=ay_2$. It is obvious that if $S$ is a subsemigroup of~$\Self\Omega$, then~$a\sim b$ implies that~$a$ and~$b$ have same kernel and same range. Furthermore, in case $S=\Selfltwo\Omega$, it is easy to verify that the converse holds (first treat left and right divisibility separately, then join the two results). In addition, $a\sim b$ in~$\Selfltwo\Omega$ implies that $\eps(a)\sim\eps(b)$ in~$\Self\Gamma$.

We shall apply this to the maps~$e$ and $f=\begin{pmatrix}x&y\end{pmatrix}\circ e$ (where, as said above, $\set{x,y}=\rng e$). Observe that $f^2=e$ and $e\sim f$; hence $\eps(f)^2=\eps(e)$ and $\eps(e)\sim\eps(f)$, so $\Ker\eps(e)=\Ker\eps(f)$ and $\rng\eps(e)=\rng\eps(f)$. We shall evaluate the map~$\eps(f)$ on each $\Ker\eps(e)$-block, that is, on each block of the decomposition
 \begin{equation}\label{Eq:Gam2}
  \Gamma=\bigsqcup_{v\in\rng\eps(e)}[v]_{\Ker\eps(e)}
 =\bigsqcup_{v\in\mu(\one)}[v]_{\Ker\eps(e)}\sqcup[t]_{\Ker\eps(e)}\,.
 \end{equation}
{}From $\mu(\one)=\rng\eps(k_x)$ and $k_x\circ g=k_x$ it follows that $\eps(g)\circ\eps(k_x)=\eps(k_x)$ for each $g\in\Selfltwo\Omega$, thus~$\eps(g)$ fixes all the elements of~$\mu(\one)$; we shall use this in the two cases $g=e$ and $g=f$. As~$[v]_{\Ker\eps(e)}=[v]_{\Ker\eps(f)}$ for each~$v\in\mu(\one)$, it follows that each element of that block is sent to~$v$ by both maps~$\eps(e)$ and~$\eps(f)$; hence~$\eps(e)$ and~$\eps(f)$ agree on $\bigsqcup_{v\in\mu(\one)}[v]_{\Ker\eps(e)}$. As the maps~$\eps(e)$ and~$\eps(f)$ have same kernel and same range, they also agree on~$[t]_{\Ker\eps(e)}$. Therefore, $\eps(e)=\eps(f)$, and thus $e=f$, a contradiction.
\end{proof}

Pick an element $\infty\in\Omega$ and set $\Omega^*=\Omega\setminus\set{\infty}$. We put
 \begin{equation}\label{Eq:defthetaZ}
 \theta_Z=\setm{(x,y)\in\Omega\times\Omega}{x\in Z\Leftrightarrow y\in Z},
 \quad \text{for each }Z\subseteq\Omega.
 \end{equation}
If $Z$ belongs to $\Pow(\Omega)\setminus\set{\es,\Omega}$, then the equivalence relation~$\theta_Z$ has exactly the two classes~$Z$ and $\Omega\setminus Z$. This holds, in particular, for each nonempty subset~$Z$ of~$\Omega^*$. In addition, $\theta_X$ and~$\theta_Y$ are distinct elements in~$\Eqtwo\Omega$, for all distinct nonempty subsets~$X$ and~$Y$ of~$\Omega^*$, so, by Lemma~\ref{L:mumeet}, we get
$\mu(\theta_X)\cap\mu(\theta_Y)=\mu(\one)$. Furthermore, it follows from Lemma~\ref{L:muantemb} that $\mu(\theta_X)$ properly contains~$\mu(\one)$, and so the family
$\bigl(\mu(\theta_X)\setminus\mu(\one)\mid X\in\Pow(\Omega^*)\setminus\set{\es}\bigr)$ is a partition of some subset of~$\Gamma$. In particular, by using Lemmas~\ref{L:cardKges2} and~\ref{L:jumpeka}, we obtain
 \[
 \card\Gamma\ges\card\mu(\one)+
 2\cdot\card\bigl(\Pow(\Omega^*)\setminus\set{\es}\bigr)\ges
 2+2\cdot(2^{\card\Omega-1}-1)=2^{\card\Omega}\,.
 \]
This concludes the proof of Theorem~\ref{T:Selfmaps1}.

\begin{corollary}\label{C:Selfmaps2}
Let~$\Omega$ and~$\Gamma$ be sets with $\card\Omega\ges2$. Then the following are equivalent:
\begin{enumerate}
\item There exists a semigroup embedding from $\Selfltwo\Omega$ into $(\Self\Gamma)^\op$.

\item There exists a monoid embedding from~$\Self\Omega$ into~$(\Self\Gamma)^\op$.

\item $\card\Gamma\ges2^{\card\Omega}$.
\end{enumerate}
\end{corollary}

\begin{proof}
(ii)$\Rightarrow$(i) is trivial, and (i)$\Rightarrow$(iii) follows from Theorem~\ref{T:Selfmaps1}. Finally, we observed (iii)$\Rightarrow$(ii) at the beginning of Section~\ref{S:Endomaps}.
\end{proof}

As $\PSelf\Omega$ embeds into $\Self(\Omega\cup\set{\infty})$ (for any element~$\infty\notin\Omega$) and, in case $\card\Omega\ges2$, the inequality $2^{\card\Omega}>\card\Omega+1$ holds, the following corollary answers simultaneously Questions~\ref{Qu:1} and~\ref{Qu:2} in the negative.

\begin{corollary}\label{C:Schein}
There is no semigroup embedding from~$\Self\Omega$ into~$(\PSelf\Omega)^\op$, for any set~$\Omega$ with at least two elements.
\end{corollary}

\section{Subspace lattices of vector spaces}\label{S:SubVectSp}

The central idea of the present section is to study how large can be a set~$I$ such that the semilattice $(\Fin{I},\cap)$ embeds into various semilattices obtained from a vector space, and then to apply this to embeddability problems of subspace posets.

We start with an easy result.

\begin{proposition}\label{P:RkVecSp}
For a set~$I$ and a right vector space~$V$ over a division ring~$K$, the following are equivalent:
\begin{enumerate}
\item $(\Fin{I},\cup,\cap,\es)$ embeds into~$(\Subdf V,+,\cap,\set{0})$;
\item $(\Fin{I},\cap)$ embeds into~$(\Sub V,\cap)$;
\item $\card I\les\dim V$.
\end{enumerate}
\end{proposition}

\begin{proof}
(i)$\Rightarrow$(ii) is trivial.

Suppose that (ii) holds, \emph{via} an embedding $\varphi\colon(\Fin{I},\cap)\hookrightarrow(\Sub V,\cap)$, and pick $e_i\in\varphi(\set{i})\setminus\varphi(\es)$, for any $i\in I$. If~$J$ is a finite subset of~$I$, $i\in I\setminus J$, and~$e_i$ is a linear combination of $\setm{e_j}{j\in J}$, then~$e_i$ belongs to $\varphi(\set{i})\cap\varphi(J)=\varphi(\es)$, a contradiction; hence $\famm{e_i}{i\in I}$ is linearly independent, and so $\card I\les\dim V$.

Finally suppose that~(iii) holds. There exists a linearly independent family\linebreak $\famm{e_i}{i\in I}$ of elements in~$V$. Define~$\varphi(X)$ as the span of $\setm{e_i}{i\in X}$, for every $X\in\nobreak\Fin{I}$. Then~$\varphi$ is an embedding from
$(\Fin{I},\cup,\cap,\es)$ into $(\Subdf V,+,\cap,\set{0})$.
\end{proof}

For embeddability of $\Fin{I}$ into $(\Sub V,+)$, we will need further results about the dimension of dual spaces. It is an old but nontrivial result that the dual~$V^*$ (i.e., the space of all linear functionals) of an infinite-dimensional vector space~$V$ is never isomorphic to~$V$. This follows immediately from the following sharp estimate of the dimension of the dual space (which is a left vector space) given in the Proposition on Page~19 in \cite[Section~II.2]{Baer}.

\goodbreak

\begin{theorem}[R. Baer, 1952]\label{T:Baer}
Let~$V$ be a right vector space over a division ring~$K$.
\begin{enumerate}
\item If~$V$ is finite-dimensional, then $\dim V^*=\dim V$.

\item If~$V$ is infinite-dimensional, then $\dim V^*=(\card K)^{\dim V}$.
\end{enumerate}
\end{theorem}

Strictly speaking, the result above is stated in~\cite{Baer} for a vector space over a \emph{field}, but the proof presented there does not make any use of the commutativity of~$K$ so we state the result for division rings. Also, we emphasize that this proof is non-constructive, in particular it uses Zorn's Lemma. Of course, replacing `right' by `left' in the statement of Theorem~\ref{T:Baer} gives an equivalent result.

By using Baer's Theorem together with some elementary linear algebra, we obtain the following result.

\begin{proposition}\label{P:Subv+rk}
For a set~$I$ and an infinite-dimensional right vector space~$V$ over a division ring~$K$, the following are equivalent:
\begin{enumerate}
\item $(\Fin{I},\cup,\cap,\es)$ embeds into $(\Subuf V,\cap,+,V)$;

\item $(\Fin{I},\cap)$ embeds into $(\Sub V,+)$;

\item $\card I\les(\card K)^{\dim V}$.
\end{enumerate}
\end{proposition}

\begin{proof}
(i)$\Rightarrow$(ii) is trivial.

Suppose that (ii) holds. To every subspace~$X$ of~$V$ we can associate its \emph{orthogonal} $X^\bot=\setm{f\in V^*}{(\forall x\in X)(f(x)=0)}$, and the assignment $X\mapsto X^\bot$ defines an embedding from $(\Sub V,+)$ into $(\Sub V^*,\cap)$. It follows that $(\Fin{I},\cap)$ embeds into~$(\Sub V^*,\cap)$. Therefore, by applying Proposition~\ref{P:RkVecSp} to the \emph{left} $K$-vector space~$V^*$, we obtain, using Theorem~\ref{T:Baer}, that $\card I\les\dim V^*=(\card K)^{\dim V}$.

Finally suppose that (iii) holds. By Theorem~\ref{T:Baer}, there exists a linearly independent family $\famm{\ell_i}{i\in I}$ of~$V^*$ (indexed by~$I$). We put
$\varphi(X)=\bigcap_{i\in X}\ker\ell_i$, for every~$X\in\Fin{I}$ (with the convention that~$\varphi(\es)=V$). It is obvious that~$\varphi$ is a homomorphism from $(\Fin{I},\cup,\es)$ to $(\Subuf V,\cap,V)$.

For every finite subset~$X$ of~$I$, if the linear map~$\ell_X\colon V\to K^X$, $v\mapsto\famm{\ell_i(v)}{i\in X}$ were not surjective, then its image would be contained in the kernel of a nonzero linear functional on~$K^X$, which would contradict the linear independence of the~$\ell_i$s; whence~$\ell_X$ is surjective. As~$\ker\ell_X=\varphi(X)$, it follows that
 \begin{equation}\label{Eq:codimphiX}
 \codim\varphi(X)=\dim K^X=\card X.
 \end{equation}
Therefore, $\varphi$ embeds $(\Fin{I},\subseteq)$ into $(\Subuf V,\supseteq)$.

Finally let $X$ and $Y$ be finite subsets of~$I$. We apply the codimension formula to the subspaces~$\varphi(X)$ and~$\varphi(Y)$, so
 \[
 \codim(\varphi(X)+\varphi(Y))+\codim(\varphi(X)\cap\varphi(Y))=
 \codim\varphi(X)+\codim\varphi(Y).
 \]
As $\varphi(X)\cap\varphi(Y)=\varphi(X\cup Y)$, an application of~\eqref{Eq:codimphiX} yields
 \[
 \codim(\varphi(X)+\varphi(Y))=\card X+\card Y-\card(X\cup Y)=\card(X\cap Y)
 =\codim\varphi(X\cap Y).
 \]
As $\varphi(X\cap Y)$ is finite-codimensional and contains $\varphi(X)+\varphi(Y)$, it follows that $\varphi(X)+\varphi(Y)=\varphi(X\cap Y)$. Therefore, $\varphi$ is as desired.
\end{proof}

We obtain the following theorem.

\begin{theorem}\label{T:NoLattEmb}
Let~$V$ and~$W$ be right vector spaces over respective division rings~$K$ and~$F$, with~$V$ infinite-dimensional. If there exists an embedding from $(\Subuf V,+)$ into $(\Sub W,\cap)$, then $\dim W\ges(\card K)^{\dim V}$.
\end{theorem}

Of course, taking~$W=V^*$ and sending every subspace~$X$ of~$V$ to its orthogonal~$X^\bot$, we see that the bound $(\card K)^{\dim V}$ is optimal.

\begin{proof}
Put $\kappa=(\card K)^{\dim V}$. It follows from Proposition~\ref{P:Subv+rk} that $(\Fin{\kappa},\cap)$ embeds into $(\Subuf V,+)$. Hence, by assumption, $(\Fin{\kappa},\cap)$ embeds into $(\Sub W,\cap)$, which, by Proposition~\ref{P:RkVecSp}, implies that $\kappa\les\dim W$.
\end{proof}

\begin{corollary}\label{C:NoLattEmb}
Let~$V$ be an infinite-dimensional vector space over any division ring. Then there is no embedding from $(\Subuf V,+)$ into $(\Sub V,\cap)$.
\end{corollary}

\begin{remark}\label{Rk:meet2+}
The statement obtained by exchanging~$\cap$ and~$+$ in Corollary~\ref{C:NoLattEmb} does not hold as a rule. Indeed, let~$V$ be an infinite-dimensional vector space, say with basis~$I$, over a division ring~$F$, and assume that $\card F\les\card I$. Now~$\Sub V$ is a meet-subsemilattice of $(\Pow(V),\cap)$, which (using complementation) is isomorphic to $(\Pow(V),\cup)$, which (as $\card V=\card I$) is isomorphic to $(\Pow(I),\cup)$, which embeds into $(\Sub V,+)$ (to each subset of~$I$ associate its span in~$V$): so $(\Sub V,\cap)$ embeds into $(\Sub V,+)$.
\end{remark}

\section{Endomorphism monoids of vector spaces}\label{S:EndVectSp}

Let $V$ be an infinite-dimensional vector space, with basis~$I$, over a division ring~$F$. Assume, in addition, that $\card F<2^{\card I}$. If $\End V$ embeds into $(\End V)^\op$, then, as $\Self I$ embeds into~$\End V$ and~$\End V$ is a submonoid of~$\Self V$, it follows from Corollary~\ref{C:Selfmaps2} that $2^{\card I}\les\card V$, a contradiction as $\card V=\card F+\card I<2^{\card I}$ (see also the proof of Theorem~\ref{T:popular}). In the present section we shall get rid of the cardinality assumption $\card F<2^{\card I}$. The special algebraic properties of vector spaces used here will be further amplified from Section~\ref{S:CSM} on, giving, for instance, related results for $G$-sets (Corollary~\ref{C:NoEmbGsets}) and modules over n\oe therian rings (Corollary~\ref{C:Srkmodule}).

\begin{theorem}\label{T:EndIneq}
Let $V$ and $W$ be infinite-dimensional vector spaces over division rings~$K$ and~$F$, respectively. If there exists a semigroup embedding from $\Endf V$ into $(\End W)^\op$, then $\dim W\ges(\card K)^{\dim V}$.
\end{theorem}

Of course, taking~$W=V^*$ and sending every endomorphism to its transpose, we see that the bound $(\card K)^{\dim V}$ is optimal.

Denote our semigroup embedding by $\eps\colon\Endf V\hookrightarrow(\End W)^\op$. We start as in the proof of Theorem~\ref{T:Selfmaps1}.

\begin{lemma}\label{L:epskerim2}
$\ker f\subseteq\ker g$ implies that $\rng\eps(g)\subseteq\rng\eps(f)$, for all $f,g\in\Endf V$.
\end{lemma}

\begin{proof}
There exists $h\in\Endf V$ such that $g=h\circ f$. Thus $\eps(g)=\eps(f)\circ\eps(h)$ and the conclusion follows.
\end{proof}

Lemma~\ref{L:epskerim2} makes it possible to define a map $\mu\colon\Subuf V\to\Sub W$ by the rule $\mu(\ker f)=\rng\eps(f)$, for each $f\in\Endf V$.

\begin{lemma}\label{L:muantemb2}
$X\subseteq Y$ if{f} $\mu(Y)\subseteq\mu(X)$, for all $X,Y\in\Subuf V$.
\end{lemma}

\begin{proof}
The direction from the left to the right follows from Lemma~\ref{L:epskerim2}. Now assume that $\mu(Y)\subseteq\mu(X)$. There are idempotent $f,g\in\Endf V$ such that $X=\ker f$ and $Y=\ker g$. As $\rng\eps(g)\subseteq\rng\eps(f)$ and~$\eps(f)$ is idempotent, $\eps(f)\circ\eps(g)=\eps(g)$, that is, $\eps(g\circ f)=\eps(g)$, and thus, as~$\eps$ is one-to-one, $g\circ f=g$, and therefore~$\ker f\subseteq\ker g$.
\end{proof}

\begin{lemma}\label{L:mu(X+Y)}
$\mu(X+Y)=\mu(X)\cap\mu(Y)$, for all $X,Y\in\Subuf V$.
\end{lemma}

\begin{proof}
Put $Z=X\cap Y$ and let $X'$, $Y'$, $T$ be subspaces of~$V$ such that $X=Z\oplus X'$, $Y=Z\oplus Y'$, and $(X+Y)\oplus T=V$. It follows that $V=Z\oplus X'\oplus Y'\oplus T$. Let~$f$ and~$g$ denote the projections of~$V$ onto $Y'\oplus T$ and $X'\oplus T$, respectively, with kernels~$X$ and~$Y$, respectively. Then $g\circ f$ is the projection of~$V$ onto~$T$ with kernel~$X+Y$.

Let $x\in\mu(X)\cap\mu(Y)$. This means that~$x$ belongs to both~$\rng\eps(f)$ and~$\rng\eps(g)$, hence, as both~$\eps(f)$ and~$\eps(g)$ are idempotent, that it is fixed by both these maps, hence that it is fixed by their composite,
$\eps(f)\circ\eps(g)=\eps(g\circ f)$, hence it lies in the range of that composite, which, as~$\ker(g\circ f)=X+Y$, is~$\mu(X+Y)$.

So we have proved that $\mu(X)\cap\mu(Y)$ is contained in~$\mu(X+Y)$.
As the converse inequality follows from Lemma~\ref{L:epskerim2}, the conclusion follows.
\end{proof}

Now Theorem~\ref{T:EndIneq} follows immediately from Theorem~\ref{T:NoLattEmb}.

Observe the contrast with the case where~$V$ is \emph{finite-dimensional} and~$K$ is \emph{commutative}: in this case, $V$ is isomorphic to its dual vector space~$V^*$, and transposition defines an \emph{isomorphism} from~$\End V$ onto~$\End V^*$.

\begin{corollary}\label{C:EndIneq2}
Let~$V$ be an infinite-dimensional vector space over any division ring. Then there is no semigroup embedding from $\Endf V$ into $(\End V)^\op$.
\end{corollary}

\begin{corollary}\label{C:EndIneq}
Let $\Omega$ be an infinite set and let $V$ be a vector space over a division ring. If $\Selffin\Omega$ has a semigroup embedding into $(\End V)^\op$, then $\dim V\ges2^{\card\Omega}$.
\end{corollary}

\begin{proof}
Denote by~$\FF_2$ the two-element field. Apply Theorem~\ref{T:EndIneq} to the $\FF_2$-vector space~$(\FF_2)^{(\Omega)}$ with basis~$\Omega$ instead of~$V$, and~$V$ instead of~$W$. We obtain that if there exists a semigroup embedding from~$\Endf\bigl((\FF_2)^{(\Omega)}\bigr)$ into~$(\End V)^\op$, then $\dim V\ges2^{\card\Omega}$. Now observe that as~$\FF_2$ is finite, $\Endf\bigl((\FF_2)^{(\Omega)}\bigr)$ is a subsemigroup of~$\Selffin((\FF_2)^{(\Omega)}\bigr)$. As~$\Omega$ and~$(\FF_2)^{(\Omega)}$ have the same cardinality, our result follows.
\end{proof}

\section{Endomorphism monoids of free algebras}\label{S:FreeAlg}

Most popular varieties of algebras have a finite similarity type (i.e., set of fundamental operations). Our next result deals with the embeddability problem for such varieties (and some more). For a variety~$\mathcal{V}$ of algebras, we shall denote by $\rF_{\mathcal{V}}(X)$ the free algebra in~$\mathcal{V}$ on~$X$. We say that~$\mathcal{V}$ is \emph{trivial} if the universe of any member of~$\mathcal{V}$ is a singleton.

\begin{theorem}\label{T:popular}
Let $\mathcal{V}$ be a nontrivial variety of algebras with similarity type~$\Sigma$. Then there is no semigroup embedding from $\End\rF_{\mathcal{V}}(\Omega)$ into\linebreak $(\End \rF_{\mathcal{V}}(\Omega))^\op$, for every infinite set~$\Omega$ such that $\card\Sigma<2^{\card\Omega}$.
\end{theorem}

\begin{proof}
Suppose that there is a semigroup embedding from $\End\rF_{\mathcal{V}}(\Omega)$ into\linebreak $(\End\rF_{\mathcal{V}}(\Omega))^\op$. As~$\mathcal{V}$ is nontrivial and every endomap of~$\Omega$ extends to a unique endomorphism of~$\rF_{\mathcal{V}}(\Omega)$, $\Self\Omega$ embeds into $\End\rF_{\mathcal{V}}(\Omega)$. As the latter is a submonoid of $\Self\rF_{\mathcal{V}}(\Omega)$, we obtain that $\Self\Omega$ embeds into $(\Self\rF_{\mathcal{V}}(\Omega))^\op$, so, by Theorem~\ref{T:Selfmaps1}, we obtain that $\card\rF_{\mathcal{V}}(\Omega)\ges2^{\card\Omega}$. However,
$\card\rF_{\mathcal{V}}(\Omega)\les \card \Omega+\card\Sigma+\aleph_0<2^{\card\Omega}$,
a contradiction.
\end{proof}

Observe that the context of Theorem~\ref{T:popular} covers most examples of algebras provided in \cite[Section~2.1]{burris}.

Our next result will show that the cardinality bound $\card\Sigma<2^{\card\Omega}$ in Theorem~\ref{T:popular} is optimal. For a monoid~$M$, an \emph{$M$-act} is a nonempty set~$X$ endowed with a map ($M\times X\to X$, $(\alpha,x)\mapsto\alpha\cdot x$) such that $1\cdot x=x$ and $\alpha\cdot(\beta\cdot x)=(\alpha\beta)\cdot x$ for all $\alpha,\beta\in M$ and all $x\in X$. Hence the similarity type of~$M$-acts consists of a collection, indexed by~$M$, of unary operation symbols. Furthermore, the free $M$-act on a set~$\Omega$, denoted by $\rF_M(\Omega)$,  can be identified with $M\times\Omega$, endowed with the `inclusion' map ($\Omega\hookrightarrow M\times\Omega$, $p\mapsto(1,p)$), and the multiplication defined by $\alpha\cdot(\beta,p)=(\alpha\beta,p)$.

For any set~$\Omega$, we shall consider the monoid~$\Rel\Omega$ of all binary relations on~$\Omega$, endowed with the composition operation defined by
 \begin{equation}\label{Eq:DefCompRel}
 \alpha\circ\beta=\setm{(x,y)\in\Omega\times\Omega}{(\exists z\in\Omega)
 ((x,z)\in\beta\text{ and }(z,y)\in\alpha)},
 \end{equation}
for all $\alpha,\beta\in\Rel\Omega$. The right hand side of~\eqref{Eq:DefCompRel} is denoted in many references by $\beta\circ\alpha$, however this conflicts with the notation $g\circ f$ for composition of functions, where every function is identified with its graph; as both composition operations will be needed in the proof, we choose to identify them. This should not cause much confusion as the monoid~$\Rel\Omega$ is \emph{self-dual}, that is, it has a dual automorphism. The latter is the \emph{transposition map} $\alpha\mapsto\alpha^{-1}$, where
 \[
 \alpha^{-1}=\setm{(x,y)\in\Omega\times\Omega}{(y,x)\in\alpha},
 \text{ for any }\alpha\in\Rel\Omega.
 \]

\begin{theorem}\label{T:AlmostDual}
Let~$\Omega$ be an infinite set and put $M=\Rel\Omega$. Then the monoid $\End\rF_M(\Omega)$ has a dual embedding.
\end{theorem}

\begin{proof}
The strategy of the proof will be the following:
\begin{enumerate}
  \item prove that for every monoid $M$ and every infinite set $\Omega$, the monoid $M^\op$ embeds in $\End\rF_M(\Omega)$; therefore $M\hookrightarrow (\End\rF_M(\Omega))^\op$;
  \item in case $M=\Rel\Omega$, prove that $\End\rF_M(\Omega)\hookrightarrow M$;
  \item items (i) and (ii) put together imply that $\End\rF_M(\Omega)\hookrightarrow(\End\rF_M(\Omega))^\op$.
\end{enumerate}

We start with any monoid~$M$. We put $x\cdot y=\famm{x(p)\cdot y(p)}{p\in\Omega}$ for any $x,y\in M^\Omega$, and we endow $\rE(M)=(\Self\Omega)\times M^\Omega$ with the multiplication given by
 \[
 (\alpha,x)\cdot(\beta,y)=(\alpha\beta,y\cdot(x\circ\beta)),
 \text{ for all }(\alpha,x),(\beta,y)\in\rE(M).
 \]
Each $(\alpha,x)\in\rE(M)$ defines an endomorphism $f_{(\alpha,x)}$ of $\rF_M(\Omega)=M\times\Omega$ by the rule
 \[
 f_{(\alpha,x)}(t,p)=(t\cdot x(p),\alpha(p)),\text{ for each }(t,p)\in M\times\Omega.
 \]
It is straightforward to verify that the assignment $(\alpha,x)\mapsto f_{(\alpha,x)}$ defines an isomorphism from~$(\rE(M),\cdot)$ onto $(\End\rF_M(\Omega),\circ)$.
Furthermore,
 \begin{equation}\label{Eq:MopEmbEnd}
 M^\op\text{ has a monoid embedding into }\End\rF_M(\Omega),
 \end{equation}
namely the assignment $x\mapsto(\id_\Omega,k_x)$, where  $k_x$ denotes the constant function on~$\Omega$ with value~$x$ (as in Section \ref{S:Endomaps}).

Now we specialize to $M=\Rel\Omega$. Let~$\infty$ be an object outside~$\Omega$ and put $\ol{\Omega}=\Omega\cup\set{\infty}$. With every $\alpha\in\Rel\Omega$ we associate the binary relation~$\ol{\alpha}=\alpha\cup\set{(\infty,\infty)}$. It is obvious that the assignment $\alpha\mapsto\ol{\alpha}$ defines a monoid embedding from~$\Rel\Omega$ into~$\Rel\ol{\Omega}$.

For each $(\alpha,x)\in\rE(M)$, we define the binary relation $\eta(\alpha,x)$ on~$\Omega\times\ol{\Omega}$ by
 \[
 \eta(\alpha,x)=\setm{((p_0,q_0),(p_1,q_1))\in(\Omega\times\ol{\Omega})^2}
 {p_1=\alpha(p_0)\text{ and }(q_1,q_0)\in\ol{x(p_0)}}.
 \]
It is straightforward to verify that the map~$\eta$ defines a monoid embedding from~$\rE(M)$ into~$\Rel(\Omega\times\ol{\Omega})$. (That~$\eta$ is one-to-one follows from our precaution of having replaced~$\Omega$ by~$\ol{\Omega}$ in the definition of the map~$\eta$; indeed, as the binary relation $\ol{x(p_0)}$ always contains the pair $(\infty,\infty)$, $\eta(\alpha,x)$ determines the pair $(\alpha,x)$.)
As~$\Rel(\Omega\times\ol{\Omega})$ is isomorphic to~$\Rel\Omega$ (use any bijection from~$\Omega\times\ol{\Omega}$ onto~$\Omega$) and by~\eqref{Eq:MopEmbEnd}, it follows from the self-duality of~$\Rel\Omega$ that the monoids~$\Rel\Omega$ and $\End\rF_M(\Omega)$ embed into each other. As $M=\Rel\Omega$ is self-dual, the conclusion follows.
\end{proof}

As shows the coming Corollary~\ref{C:NoEmbGsets}, Theorem~\ref{T:AlmostDual} cannot be extended to \emph{$G$-sets} (i.e., $G$-acts), for \emph{groups}~$G$.
See also Problem~\ref{Pb:SelfDual}.

\section{C-, S-, and M-independent subsets in algebras}\label{S:CSM}

We first recall some general notation and terminology. For an algebra~$A$ (that is, a nonempty set endowed with a collection of finitary operations), we denote by~$\Sub A$ (resp., $\End A$) the collection of all subuniverses (resp., endomorphisms) of~$A$. We also denote by~$\seq{X}$ the subuniverse of~$A$ generated by a subset~$X$ of~$A$; in case $X=\set{x_1,\dots,x_n}$, we shall write $\seq{x_1,\dots,x_n}$ instead of $\seq{\set{x_1,\dots,x_n}}$. We shall also put $X\vee Y=\seq{X\cup Y}$, for all $X,Y\in\Sub A$.
A subset~$I$ of~$A$ is said to be
\begin{itemize}
\item \emph{C-independent}, if $x\notin\seq{I\setminus\set{x}}$, for all~$x\in I$;

\item \emph{M-independent}, if every map from~$I$ to~$A$ can be extended to some homomorphism from~$\seq{I}$ to~$A$.
\item \emph{S-independent}, if every map from~$I$ to~$I$ can be extended to some homomorphism from ~$\seq{I}$ to~$A$.
\end{itemize}
In these definitions, C stands for \emph{closure}, as the definition of C-independence relies upon a closure operator; M stands for Marczewski who introduced M-independence in \cite{mar1}; S stands for \'{S}wierczkowski who introduced this notion in \cite{swier}.

Say that a subset~$I$ of~$A$ is \emph{non-degenerate}, if $I\cap\seq{\es}=\es$. The following result, with straightforward proof, shows that aside from degenerate cases, M-independence implies S-independence implies C-independence. (None of the converses hold as a rule \cite{glazek}).

\begin{proposition}\label{P:Degen}
Let $I$ be a subset in an algebra $A$. The following assertions hold:
\begin{enumerate}
\item $I$ is S-independent degenerate if{f} $I$ is a singleton contained in $\seq{\es}$.

\item $I$ is M-independent degenerate if{f} $I=A=\seq{\es}$ is a singleton.

\item If $I$ is M-independent, then $I$ is S-independent.

\item If $I$ is S-independent non-degenerate, then $I$ is C-independent.
\end{enumerate}
\end{proposition}

The following result generalizes the main part of Proposition~\ref{P:RkVecSp}. It relates the existence of large either S-independent or C-independent subsets of an algebra~$A$ and the existence of meet-embeddings of large $\Fin{I}$ into the subuniverse lattice of~$A$.

\begin{proposition}\label{P:rkIndep}
The following statements hold, for every algebra~$A$ and every set~$I$:
\begin{enumerate}
\item If~$I$ is a non-degenerate S-independent subset of~$A$, then $(\Fin{I},\cup,\cap)$ embeds into $(\Sub A,\vee,\cap)$.

\item If $(\Fin{I},\cap)$ embeds into $(\Sub A,\cap)$, then~$A$ has a C-independent subset~$X$ such that $\card I\les\card X$.
\end{enumerate}
\end{proposition}

\begin{proof}
(i). Let~$I$ be a non-degenerate S-independent subset of~$A$, we shall prove that $(\Fin{I},\cup,\cap)$ embeds into $(\Sub A,\vee,\cap)$. If~$I=\es$ then the result is trivial. Suppose that~$I=\set{p}$. As~$I$ is non-degenerate, $p\notin\seq{\es}$, thus~$\seq{\es}$ is strictly contained in~$\seq{p}$, and the result follows.

Suppose from now on that~$I$ has at least two elements. We define a map $\varphi\colon\Fin{I}\to\Sub A$ by setting
 \begin{equation}\label{Eq:defphi(es)}
 \varphi(\es)=\bigcap\famm{\seq{p}}{p\in I}\,,
 \end{equation}
while $\varphi(X)=\seq{X}$ for any nonempty $X\in\Fin{I}$. It is obvious that~$\varphi$ is a \jh\ from~$\Fin{I}$ to~$\Sub A$. Suppose that $\varphi(X)\subseteq\varphi(Y)$, for $X,Y\in\Fin{I}$, and let $p\in X\setminus Y$.
Suppose first that $Y=\es$. As $X\subseteq\varphi(X)\subseteq\varphi(Y)=\varphi(\es)$ and by~\eqref{Eq:defphi(es)}, we obtain that $p\in\seq{q}$ for each $q\in I$, thus, as~$I$ is C-independent (cf. Proposition~\ref{P:Degen}), $I=\set{p}$, a contradiction.
Suppose now that $Y$ is nonempty. Let $q\in I$. As~$I$ is S-independent, there exists an endomorphism~$f$ of~$\seq{I}$ such that $f(p)=q$ and $f\res_Y=\id_Y$.
{}From $X\subseteq\varphi(X)\subseteq\varphi(Y)=\seq{Y}$ it follows that $p\in\seq{Y}$, hence $q=f(p)=p$, so $I=\set{p}$, a contradiction.

Therefore, $\varphi$ is a \je.

Now let $X,Y\in\Fin{I}$, we shall prove that $\varphi(X)\cap\varphi(Y)$ is contained in $\varphi(X\cap Y)$. So let $a\in\varphi(X)\cap\varphi(Y)$. Fix one-to-one enumerations
 \begin{align*}
 X\setminus Y&=\set{x_0,\dots,x_{k-1}},\\
 Y\setminus X&=\set{y_0,\dots,y_{l-1}},\\
 X\cap Y&=\set{z_0,\dots,z_{n-1}}.
 \end{align*}
There are terms~$s$ and~$t$ such that
 \begin{equation}\label{Eq:2repra}
 a=s(x_0,\dots,x_{k-1},z_0,\dots,z_{n-1})=t(y_0,\dots,y_{l-1},z_0,\dots,z_{n-1}).
 \end{equation}
Suppose first that $X\cap Y\neq\es$, so $n>0$. As~$I$ is S-independent, there exists an endomorphism~$f$ of~$\seq{I}$ that fixes all~$y_i$s and all~$z_i$s such that $f(x_i)=z_0$ for each $i<k$. {}From the second equation in~\eqref{Eq:2repra} it follows that $f(a)=a$, hence, by the first equation in~\eqref{Eq:2repra},
 \[
 a=f(a)=s(\underbrace{z_0,\dots,z_0}_{k\text{ times}},z_0,\dots,z_{n-1})
 \in\varphi(X\cap Y)\,.
 \]
Now assume that $X\cap Y=\es$. By applying the case above to~$X\cup\set{p}$ and~$Y\cup\set{p}$, we obtain that $a\in\varphi(\set{p})=\seq{p}$, for each~$p\in I$. Hence, by~\eqref{Eq:defphi(es)}, $a$ belongs to~$\varphi(\es)$.

In any case, $a\in\varphi(X\cap Y)$, and so $\varphi$ is a \mh.

(ii). Let $\varphi\colon(\Fin{I},\cap)\hookrightarrow(\Sub A,\cap)$ be an embedding, and pick $e_i\in\varphi(\set{i})\setminus\varphi(\es)$, for any~$i\in I$. If $i$, $i_0$, \dots, $i_{n-1}$ are distinct indices in~$I$ and $e_i$ belongs to $\seq{e_{i_0},\dots,e_{i_{n-1}}}$, then it belongs to~$\varphi(\set{i})\cap\varphi(\set{i_0,\dots,i_{n-1}})=\varphi(\es)$, a contradiction. Therefore, the family $\famm{e_i}{i\in I}$ is C-independent.
\end{proof}

On the other hand, by mimicking the arguments used in the proofs of earlier results, we obtain the following set of results.

\begin{proposition}\label{P:NegEmbInd}
Let $A$ be an algebra, let~$\Omega$ be an infinite set, and let~$V$ be an infinite-dimensional right vector space over a division ring~$K$. Put $\kappa=(\card K)^{\dim V}$ and $\lambda=2^{\card\Omega}$. Then the following statements hold:
\begin{enumerate}
\item If $\Endf V$ has a semigroup embedding into $(\End A)^\op$, then $(\Subuf V,+)$ embeds into $(\Sub A,\cap)$.

\item If $(\Subuf V,+)$ embeds into $(\Sub A,\cap)$, then $(\Fin{\kappa},\cap)$ embeds into $(\Sub A,\cap)$.

\item If $\Selffin\Omega$ has a semigroup embedding into $(\End A)^\op$, then $(\Fin{\lambda},\cap)$ embeds into $(\Sub A,\cap)$.
\end{enumerate}
\end{proposition}

\begin{proof}
(i). Let $\eps\colon\Endf V\hookrightarrow(\End A)^\op$ be a semigroup embedding. As in the proof of Theorem~\ref {T:EndIneq}, we can construct a map $\mu\colon\Subuf V\to\Sub A$ by the rule $\mu(\ker f)=\rng\eps(f)$, for each $f\in\Endf V$. As in the proof of Theorem~\ref{T:EndIneq}, $\mu$ is an embedding from $(\Subuf V,+)$ into $(\Sub A,\cap)$.

(ii). It follows from Proposition~\ref{P:Subv+rk} that $(\Fin{\kappa},\cap)$ embeds into $(\Subuf V,+)$, thus into $(\Sub A,\cap)$.

(iii). As in the proof of Corollary~\textup{\ref{C:EndIneq}}, there exists a semigroup embedding from $\Endf\bigl((\FF_2)^{(\Omega)}\bigr)$ into $\Selffin\Omega$, and hence into $(\End A)^\op$. The conclusion follows then from~(i) and~(ii) above.
\end{proof}

\section{Embedding endomorphism semigroups of SC-ranked algebras}\label{S:IndepAlg}

In the present section we shall indicate how certain results of Sections~\ref{S:SubVectSp} and~\ref{S:EndVectSp} can be extended to more general objects, which we shall call \emph{SC-ranked algebras}.

We start by recalling the following result.

 \begin{lemma}[\cite{MKNT}, p.~50, Exercise 6]\label{L:506}
 For an algebra $A$, the following conditions are equivalent:
 \begin{enumerate}
\item [(1)] for every subset $X$ of $A$ and all elements $u,v$ of $A$, if $u\in\seq{X\cup\set{v}}$ and $u\notin\seq{X}$, then $v\in \langle X\cup\set{u}\rangle$;

\item [(2)] for every subset $X$ of $A$ and every element $u\in A$, if $X$ is C-independent and $u\notin \seq{X}$, then $X\cup \set{u}$ is C-independent;

\item [(3)] for every subset $X$ of $A$, if $Y$ is a maximal C-independent subset of $X$, then $\seq{X} =\seq{Y}$;

\item [(4)] for all subsets $X, Y$ of $A$ with $Y\subseteq X$, if $Y$ is
 C-independent, then there is a C-independent set $Z$ with $Y\subseteq Z\subseteq X$ and $\seq{Z}=\seq{X}$.
 \end{enumerate}
 \end{lemma}

An algebra $A$ is said to be a \emph{matroid algebra} if it satisfies one (and hence all) of the equivalent conditions of Lemma~\ref{L:506}.

\begin{definition}\label{D:Tbasis}
For $\rT\in\set{\rM,\rS,\rC}$, a \emph{T-basis} of an algebra~$A$ is a T-independent generating subset of $A$. We say that~$A$ is a \emph{T-algebra} if it has a T-basis.
\end{definition}

Clearly every free algebra is an M-algebra, thus an S-algebra.

\begin{definition}\label{D:TQalg}
For $\rT,\rQ\in\set{\rM,\rS,\rC}$, a \emph{TQ-algebra} is an algebra where the notions of T-independence and Q-independence coincide.
\end{definition}

The MC-algebras appear in the literature as \emph{$v^{**}$-algebras} (see \cite{nar3, ur5}). Every absolutely free algebra is an MC-algebra (see \cite{ur5} for this and many other examples).

A matroid MC-algebra is said to be an \emph{independence algebra}. These algebras attracted the attention  of experts in Universal Algebra (they were originally called \emph{$v^*$-algebras}; see \cite{ArFo,mar1,mar2,mar3,mar4,nar,nar2,nar3,ur1,ur2,ur3,ur5} and \cite{glazek} for  hundreds of references on the topic), Logic (e.g. \cite{givant1,givant2, zilber1,zilber2}) and Semigroup Theory (e.g. \cite{FoLe, FoLe2,gould}).
Familiar examples of independence algebras are sets, free $G$-sets (for a group $G$) and vector spaces (see \cite{CaSz,ur5}). Observe that independence algebras are MC-algebras and the latter are SC-algebras.

\begin{definition}\label{D:Sranked}
An algebra $A$ is said to be \emph{SC-ranked}, if it has an S-basis~$\Omega$ such that $\card X\les\card\Omega$ for each C-independent subset~$X$ of~$A$. The cardinality of this set~$\Omega$ is said to be the \emph{rank} of~$A$, and denoted by~$\Rank A$.
\end{definition}

By Lemma \ref{L:506}(4), every matroid S-algebra~$A$ is an SC-ranked algebra. Observe that $\Rank A$ is then the cardinality of \emph{any} C-basis of~$A$.

It should  be observed that not every SC-algebra contains a C-independent generating set (see the example following the proof of Theorem~4 in \cite[Section~32]{GrUA}).

\begin{theorem}\label{T:EmbBded}
Let~$A$ and~$B$ be SC-ranked algebras with~$\Rank A$ infinite. If there exists a semigroup embedding from $\End A$ into $(\End B)^\op$, then $\Rank B\ges 2^{\Rank A}$.
\end{theorem}

\begin{proof}
Let~$X$ be an S-basis of~$A$. Then $\Selffin X$ embeds into $\Self X$, which (as~$X$ is an S-basis) embeds into $\End A$, which embeds into $(\End B)^\op$. Therefore, by Proposition~\ref{P:NegEmbInd}(iii) combined with Proposition~\ref{P:rkIndep}(ii), there exists a C-independent set $Y\subseteq B$  with $\card Y\ges 2^{\card X}$. As~$B$ is SC-ranked, $\card Y\les\Rank B$ and the result follows.
\end{proof}

\begin{corollary}\label{C:RkA=RkBindep}
For SC-ranked algebras $A$, $B$ such that $\Rank A\ges\Rank B\ges \aleph_0$, there is no semigroup embedding from~$\End A$ into~$(\End B)^\op$. In particular, the semigroup $\End A$ has no dual embedding.
\end{corollary}

In particular, Corollary~\ref{C:RkA=RkBindep} applies to independence algebras.

The classification problem of all MC-algebras is open since the mid sixties. As Gr\"{a}tzer says ``\emph{There are some results on [the classification of MC-algebras, that is] $v^{**}$-algebras; but the problem is far from settled}'' \cite[p. 205]{GrUA}. Likewise, SC-ranked algebras are not classified; in fact, the requirement to be SC-ranked seems so weak that it seems unlikely that this could ever be done. For example, Theorems~\ref{T:FM(Omega)Srk} and~\ref{T:Srkmodule} give us, respectively, a characterization of SC-ranked free $M$-acts (for monoids~$M$) and a sufficient condition for a free module to be SC-ranked, in terms of an antichain condition of the left divisibility relation on the monoid, and a n\oe therianity condition on the ring, respectively. The corresponding classes of monoids, or rings, are so large that they are certainly beyond the reach of any classification.

Another point is that in order to obtain results such as Theorem~\ref{T:EmbBded}, the statement, for an algebra~$A$, to be SC-ranked, is a compromise between conciseness and generality. In particular, it can be further weakened (e.g., by using meet-embeddings of semilattices $\Fin{I}$ into subuniverse lattices), and it seems likely that more algebras would satisfy the possible weakenings of SC-rankedness, although it is unclear whether there would be any `natural' such example.

In Sections~\ref{S:SrkfreeMsets} and~\ref{S:SrkedMod}, we shall illustrate the notion of SC-rankedness on \emph{$M$-acts} and \emph{modules}.

\section{SC-ranked free $M$-acts}\label{S:SrkfreeMsets}

In the present section, we shall characterize SC-ranked free $M$-acts (cf. Section~\ref{S:FreeAlg}).

In any monoid~$M$, we define preorderings~$\utrl$ and~$\utrr$ by the rule
 \[
 u\utrl v\Leftrightarrow(\exists t)(v=tu)\,,\quad
 u\utrr v\Leftrightarrow(\exists t)(v=ut)\,,\quad\text{for all }u,v\in M.
 \]
We say that~$M$ is \emph{left uniserial}, if~$\utrl$ is a total preordering, that is, for any elements $u,v\in M$, either $u\utrl v$ or $v\utrl u$. This occurs, in particular, in the somehow degenerate case where~$M$ is a \emph{group}.

\begin{theorem}\label{T:FM(Omega)Srk}
Let~$M$ be a monoid and let~$\Omega$ be a nonempty set. Then $\rF_M(\Omega)$ is SC-ranked if{f} either~$\Omega$ is finite and~$M$ is left uniserial, or~$\Omega$ is infinite and every $\utrl$-antichain of~$M$ has at most~$\card\Omega$ elements.
\end{theorem}

\begin{proof}
We shall repeatedly use the easily verified fact that the C-independent subsets of~$\rF_M(\Omega)$ are exactly the subsets~$Y$ such that
$Y\cdot p^{-1}=\setm{u\in M}{u\cdot p\in Y}$ is a $\utrl$-antichain for every~$p\in\Omega$. Observe also that~$\Omega$ is an M-basis, thus an S-basis, of~$\rF_M(\Omega)$.

Suppose first that~$M$ has a $\utrl$-antichain~$U$ such that $\card\Omega<\card U$. Pick $p\in\Omega$. Observe that $U\cdot p=\setm{u\cdot p}{u\in U}$ is a C-independent subset of~$\rF_M(\Omega)$ of cardinality greater than~$\card\Omega$. As~$\Omega$ is an S-basis of~$\rF_M(\Omega)$, it follows that~$\rF_M(\Omega)$ is not SC-ranked.

Now suppose that~$M$ is not left uniserial and~$\Omega$ is finite. Let $u,v\in M$ be $\utrl$-incomparable. Then the subset $\setm{u\cdot p}{p\in\Omega}\cup\setm{v\cdot p}{p\in\Omega}$ is a C-independent subset of~$\rF_M(\Omega)$ with cardinality~$2\cdot\card\Omega$, so again~$\rF_M(\Omega)$ is not SC-ranked.

If~$M$ is left uniserial, then the C-independent subsets of~$\rF_M(\Omega)$ are exactly the subsets of the form $\setm{f(p)\cdot p}{p\in X}$, for a subset~$X$ of~$\Omega$ and a map~$f\colon X\to M$. Hence every C-independent subset has at most $\card\Omega$ elements, and so~$\rF_M(\Omega)$ is SC-ranked.

Finally assume that~$\Omega$ is infinite and that every $\utrl$-antichain of~$M$ has cardinality at most~$\card\Omega$. For every C-independent subset~$Y$ of~$\rF_M(\Omega)$ and every $p\in\Omega$, the subset~$Y\cdot p^{-1}$ is a~$\utrl$-antichain of~$M$, thus it has cardinality below~$\card\Omega$; hence, as~$\Omega$ is infinite, $\card Y\les\card\Omega$. Therefore, $\rF_M(\Omega)$ is SC-ranked.
\end{proof}

As an immediate consequence of Corollary~\ref{C:RkA=RkBindep} and Theorem~\ref{T:FM(Omega)Srk}, we observe the following.

\begin{corollary}\label{C:NoEmbSCRM}
Let~$M$ be a monoid and let~$\Omega$ be an infinite set. If every $\utrl$-antichain of~$M$ has at most~$\card\Omega$ elements, then the semigroup $\End\rF_M(\Omega)$ has no dual embedding.
\end{corollary}

Observe that $\rF_M(\Omega)$ is almost never a matroid algebra:

\begin{proposition}\label{P:FM(Omega)matroid}
Let~$M$ be a monoid and let~$\Omega$ be a nonempty set. Then $\rF_M(\Omega)$ is a matroid algebra if{f}~$M$ is a group.
\end{proposition}

\begin{proof}
If~$M$ is a group, then it is straightforward to verify that $\rF_M(\Omega)$ satisfies Condition~(1) of Lemma~\ref{L:506}, so it is a matroid algebra.

Conversely, suppose that $\rF_M(\Omega)$ is a matroid algebra. Let $u\in M$ and pick $p\in\Omega$. {}From $u\cdot p\in\seq{1\cdot p}\setminus\seq{\es}$ and the matroid condition it follows that $1\cdot p\in\seq{u\cdot p}$, that is, $u$ is left invertible in~$M$. As this holds for all $u\in M$, $M$ is a group.
\end{proof}

The following result gives us a wide range of MC-algebras that are usually not SC-ranked. Denote by~$X^*$ the free monoid on~$X$, for any set~$X$.

\begin{proposition}\label{P:MCnotSrk}
Let~$\Omega$ and~$X$ be sets, with~$\Omega$ nonempty. Then~$\rF_{X^*}(\Omega)$ is both an M-algebra and an MC-algebra.
\end{proposition}

\begin{proof}
As~$\Omega$ is an M-basis of~$\rF_{X^*}(\Omega)$, the latter is an M-algebra.

Now let~$Y$ be a C-independent subset of~$\rF_{X^*}(\Omega)$. This means that $Y\cdot p^{-1}$ is a $\utrl$-antichain of~$X^*$ for each $p\in\Omega$. Now let $f\colon Y\to\rF_{X^*}(\Omega)$ be any mapping. Consider pairs $(t_0,y_0)$ and $(t_1,y_1)$ of $X^*\times Y$ such that $t_0y_0=t_1y_1$. This means that there are $p\in\Omega$ and $u_0,u_1\in X^*$ such that $y_0=u_0\cdot p$, $y_1=u_1\cdot p$, and $t_0u_0=t_1u_1$. As~$X^*$ is the free monoid on~$X$, either $t_1\utrr t_0$ or $t_0\utrr t_1$; suppose, for example, that the first case holds, so $t_0=t_1w$ for some~$w\in X^*$. {}From $t_1wu_0=t_0u_0=t_1u_1$ it follows that $wu_0=u_1$, thus $u_0\utrl u_1$, hence, as $Y\cdot p^{-1}$ is a $\utrl$-antichain, $u_0=u_1$, and hence $y_0=y_1$ and $t_0=t_1$. Therefore, there exists a unique map $\ol{f}\colon\seq{Y}\to\rF_{X^*}(\Omega)$ such that $\ol{f}(t\cdot y)=t\cdot f(y)$ for each $(t,y)\in X^*\times Y$. Clearly, $\ol{f}$ is a morphism, and so $\rF_{X^*}(\Omega)$ is an MC-algebra.
\end{proof}

Observe that $X$ is a $\utrl$-antichain of~$X^*$. Hence, by Theorem~\ref{T:FM(Omega)Srk}, if $\card X>\card\Omega$, then $\rF_{X^*}(\Omega)$ is not SC-ranked, although, by Proposition~\ref{P:MCnotSrk}, it is both an M-algebra and an MC-algebra.

As a particular case of Corollary~\ref{C:NoEmbSCRM}, we obtain

\begin{corollary}\label{C:NoEmbGsets}
Let $\Omega$ be an infinite set and let~$G$ be a group. Then $\End\rF_G(\Omega)$ has no dual embedding.
\end{corollary}

Corollary~\ref{C:NoEmbGsets} does not extend to \emph{$M$-acts} (for a monoid~$M$), see Theorem~\ref{T:AlmostDual}.

\section{SC-ranked free modules and $\kappa$-n\oe therianity}\label{S:SrkedMod}

In this section, all modules will be left modules over (unital, associative) rings.

\begin{definition}\label{D:gknoeth}
Let~$\kappa$ be a regular cardinal. A module~$M$ is \emph{$\kappa$-n\oe therian}, if every increasing $\kappa$-sequence of submodules of~$M$ is eventually constant.
\end{definition}

In particular, $M$ is n\oe therian if{f} it is $\aleph_0$-n\oe therian. For a regular cardinal~$\kappa$, $M$ is $\kappa$-n\oe therian if{f} there is no strictly increasing $\kappa$-sequence of submodules of~$M$. Hence, if $\kappa<\lambda$ are regular cardinals and~$M$ is $\kappa$-n\oe therian, then~$M$ is also~$\lambda$-n\oe therian.

C-independent subsets and $\kappa$-n\oe therian modules are related as follows.

\begin{lemma}\label{L:noeth2ind}
Let~$\kappa$ be a regular cardinal. If a module~$M$ is $\kappa$-n\oe therian, then every C-independent subset of~$M$ has cardinality smaller than~$\kappa$.
\end{lemma}

\begin{proof}
Suppose that there exists a C-independent subset $\setm{x_\xi}{\xi<\kappa}$ of~$M$, where $\xi\mapsto x_\xi$ is one-to-one. The family $\famm{X_\alpha}{\alpha<\kappa}$, where~$X_\alpha$ is the submodule generated by $\setm{x_\xi}{\xi<\alpha}$, is a strictly increasing $\kappa$-sequence of submodules of~$M$, a contradiction.
\end{proof}

\begin{lemma}\label{L:FinSumNoeth}
Let $\kappa$ be a regular cardinal and let~$M$ be a module. Then any finite sum of $\kappa$-n\oe therian submodules of~$M$ is $\kappa$-n\oe therian.
\end{lemma}

\begin{proof}
As the proof of the (classical) result that the sum of two n\oe therian modules is n\oe therian (i.e., the case where $\kappa=\aleph_0$), see, for example, the Corollary in \cite[Section~VI.1]{Lang}.
\end{proof}

\begin{lemma}\label{L:InfSumNoeth}
Let~$\kappa$ be a regular cardinal, let $M$ be a module, and let $\famm{M_i}{i\in I}$ be a family of $\kappa$-n\oe therian submodules of~$M$ such that $\card I<\kappa$. Then the sum $\sum_{i\in I}M_i$ is $\kappa$-n\oe therian.
\end{lemma}

\begin{proof}
We put $M_J=\sum_{i\in J}M_i$, for each $J\subseteq I$.
Let $\famm{X_\xi}{\xi<\kappa}$ be an increasing $\kappa$-sequence of submodules of~$M_I$.
For every $J\in\Fin{I}$, it follows from Lemma~\ref{L:FinSumNoeth} that there exists $\alpha_J<\kappa$ such that $X_\xi\cap M_J=X_{\alpha_J}\cap M_J$ for each $\xi\ges\alpha_J$. As~$\kappa$ is regular and greater than $\card\bigl(\Fin{I}\bigr)$, the supremum~$\alpha=\bigvee\famm{\alpha_J}{J\in\Fin{I}}$ is smaller than~$\kappa$. Observe that $X_\xi=X_\alpha$ for each $\xi\ges\alpha$.
\end{proof}

We shall use the standard convention to denote by $\RR$ the ring~$R$ viewed as a left module over itself, for any ring~$R$. For a regular cardinal~$\kappa$, we say that~$R$ is \emph{left $\kappa$-n\oe therian}, if the module~$\RR$ is $\kappa$-n\oe therian.

For a module~$M$ and a set $\Omega$, we denote by~$M^{(\Omega)}$ the module of all families $\famm{x_p}{p\in\Omega}\in M^\Omega$ such that $\setm{p\in\Omega}{x_p\neq0}$ is finite. In particular, $\RR^{(\Omega)}$ is the free left $R$-module on~$\Omega$.

We denote by $\kappa^+$ the successor cardinal of a cardinal~$\kappa$.

\begin{proposition}\label{P:SrankedMod}
Let $\Omega$ be an infinite set and let $R$ be a left $(\card\Omega)^+$-n\oe therian ring. Then the free module $\RR^{(\Omega)}$ is SC-ranked.
\end{proposition}

This makes it possible to produce many SC-ranked modules.

\begin{theorem}\label{T:Srkmodule}
Let $\kappa$ be an infinite cardinal and let~$R$ be a left $\kappa^+$-n\oe therian ring. Then the free left module $\RR^{(\Omega)}$ is SC-ranked, for every set~$\Omega$ such that $\card\Omega\ges\kappa$.
\end{theorem}

\begin{proof}
Put $\lambda=\card\Omega$. Of course, $\Omega$ is an S-basis of $\RR^{(\Omega)}$. As, by Lemma~\ref{L:InfSumNoeth}, $\RR^{(\Omega)}$ is a $\lambda^+$-n\oe therian left module, it follows from Lemma~\ref{L:noeth2ind} that every C-independent subset of~$\RR^{(\Omega)}$ has cardinality at most~$\lambda$.
\end{proof}

By using Corollary~\ref{C:RkA=RkBindep}, we obtain the following result.

\begin{corollary}\label{C:Srkmodule}
Let $R$ be a left $\aleph_1$-n\oe therian ring. Then the free module $\RR^{(\Omega)}$ is SC-ranked, for every infinite set~$\Omega$. Consequently, the semigroup $\End\bigl(\RR^{(\Omega)}\bigr)$ has no dual embedding.
\end{corollary}

In particular, Corollary~\ref{C:Srkmodule} applies to the case where the ring~$R$ is \emph{left n\oe therian}.

\section{Open problems}\label{S:Concl}

We observed in Remark~\ref{Rk:meet2+} that whenever~$V$ is an infinite-dimensional vector space over a division ring~$F$ such that $\card F\les\dim V$, there exists an embedding from $(\Sub V,\cap)$ into $(\Sub V,+)$. We do not know whether the cardinality restriction is necessary.

\begin{problem}\label{Pb:EmbVectSp}
Let~$V$ be an infinite-dimensional vector space over a division ring~$F$ such that $\dim V<\card F$. Does $(\Sub V,\cap)$ embed into $(\Sub V,+)$?
\end{problem}

In Theorem~\ref{T:AlmostDual}, we show that the endomorphism monoid of a free~$M$-act, for a monoid~$M$, may embed into its dual. We do not know if this can also happen for \emph{modules}:

\begin{problem}\label{Pb:EndModEmb}
Are there a unital ring~$R$ and a free left module~$F$ of infinite rank over~$R$ such that $\End F$ embeds into its dual?
\end{problem}

\begin{problem}\label{Pb:SelfDual}
Does there exist a nontrivial variety~$\mathcal{V}$ of algebras such that $\End\rF_{\mathcal{V}}(\omega)$ has a dual automorphism?
\end{problem}

By Theorem~\ref{T:popular}, the similarity type of any variety~$\mathcal{V}$ solving Problem~\ref{Pb:SelfDual} should have cardinality at least~$2^{\aleph_0}$. For a partial positive result, we refer to Theorem~\ref{T:AlmostDual}.

K. Urbanik introduces in~\cite{v_*} a subclass of the class of MC-algebras, called there \emph{$v_*$-algebras}. He also classifies these algebras in terms of modules and transformation semigroups.

Not every $v_*$-algebra has a C-basis. For example, denote by~$\ZZ_{(2)}$ the valuation ring of all rational numbers with odd denominator; then the field~$\QQ$ of all rational numbers, viewed as a $\ZZ_{(2)}$-module, is a $v_*$-algebra (cf. \cite[Section~3]{v_*}). However, for any nonzero rational numbers~$a$ and~$b$, either~$a/b$ or~$b/a$ belongs to~$\ZZ_{(2)}$, thus any C-independent subset of~$\QQ$ has at most one element. Since~$\QQ$ is not a finitely generated $\ZZ_{(2)}$-module, it has no C-basis.

\begin{problem}\label{Pb:v_*}
Let~$A$ be a $v_*$-algebra with an infinite S-basis. Can $\End A$ be embedded into its dual?
\end{problem}

By Corollary~\ref{C:RkA=RkBindep}, Problem~\ref{Pb:v_*} would have a negative answer if we could prove that every $v_*$-algebra with an infinite S-basis is also SC-ranked. However, we do not know this either.

\section*{Acknowledgment}
We thank the referee for several useful suggestions regarding the form of the paper, as well as for pointing to us the triviality of an earlier version of Problem~\ref{Pb:EmbVectSp} in the first version of the paper.

\end{document}